\documentclass{article}
\usepackage[latin1]{inputenc}
\usepackage{amsmath}
\usepackage{amssymb}
\usepackage{amsthm}
\usepackage{mathrsfs}
\usepackage{lscape}
\usepackage{euscript}
\usepackage{latexsym,enumerate,color}
\usepackage{authblk}

\allowdisplaybreaks

\begin{document}

\theoremstyle{plain} \newtheorem{Thm}{Theorem}
\newtheorem{prop}[Thm]{Proposition}
\newtheorem{lem}[Thm]{Lemma}
\newtheorem{step}[Thm]{Step}
\newtheorem{de}[Thm]{Definition}
\newtheorem{obs}[Thm]{Observation}
\newtheorem{cor}[Thm]{Corollary}
\newtheorem{remark}[Thm]{Remark}

\newcommand{\kor}{\mathring }
\newcommand{\Korz}{\mathring{\mathbb{Z}} }
\newcommand{\korx}{\mathring{x} }
\newcommand{\res}{\restriction }
\newcommand{\kalapz }{\hat{\mathbb{Z}} }
\theoremstyle{remark} \newtheorem*{pf}{Proof}
\renewcommand\theenumi{(\alph{enumi})}
\renewcommand\labelenumi{\theenumi}
\renewcommand{\qedsymbol}{}
\renewcommand{\qedsymbol}{\ensuremath{\blacksquare}}

\newcommand{\rad}{{\sf Rad}}
\def\und{\underline}
\def\Q{\mathbb{Q}}
\def\Z{\mathbb{Z}}
\def\C{\mathbb{C}}
\newcommand{\Ld}{|\!\!|}
\newcommand{\Rd}{|\!\!|}
\newcommand{\dsl}{/\!\!/}
\newcommand{\cala}{\mathcal }
\newcommand{\etha}{\alpha}
\newcommand{\embedding}{\hookrightarrow}
\newcommand{\semidirect}{\rtimes}
\newcommand{\sr}{S-ring }
\newcommand{\pp}{\mbox {\bf P}\,}
\newcommand{\pr}{\mbox {\bf Proof.\ }}
\newcommand{\aaa}{\mbox {\cala A}}
\newcommand{\pset}{\mbox {\{p_1,...,p_m\}}}
\newcommand{\zn}{\mbox {\bf Z}_n}
\newcommand{\pg}{permutation group }
\newcommand{\pgs}{permutation groups }
\newcommand{\srg}[3]{${\cala #1}=<{\underline #2}_0,...,{\underline
#2}_#3>$}
\newcommand{\isom}{\cong}
\newcommand {\sth}{ such that }
\newcommand {\imp}{ imprimitivity system }
\newcommand {\tram}{ $V(C_n,G_1)$ }
\newcommand {\pair}[1] {#1_1 #1_2^{-1}}
\newcommand {\cnp}{ C_{n/p} }
\newcommand {\normal}{\trianglelefteq}
\newcommand {\congruence} {\equiv}
\newcommand {\wreath}{\wr}
\newcommand {\longline}{\left|\right.}
\newcommand {\Sup}[1]{{\sf Sup}({#1})}
\newcommand {\Aut}[1]{{\sf Aut}({#1})}
\newcommand {\Out}[1]{{\sf Out}({#1})}
\newcommand {\Inn}[1]{{\sf Inn}({#1})}
\newcommand{\aalg}{{\sf Aut}_{alg}}
\newcommand {\New}[1]{{\sl #1}}
\newcommand {\Sym}[1]{{\sf Sym }(#1)}
\newcommand {\encodes}{\ encodes\ }
\newcommand {\tran}[1]{{#1}^*}
\newcommand {\twomin}{({\sf Sup}_2(H_R))^{min}}
\newcommand{\Log}[2]{{\mbox{\boldmath$\eta$\unboldmath}}_{#1}(#2)}
\newcommand {\Bsets}[1]{{\sf Bsets}(#1)}
\newcommand {\tranjugate}{tranjugate\ }
\newcommand {\cclosed}{ complete \ }
\newcommand{\CCI}{ {${\rm CI}^{(2)}$} }
\newcommand{\fp}{{\Zset}_p}
\newcommand{\HH}{{\sf H}}
\newcommand{\GG}{{\sf G}}
\newcommand{\m}[1]{\overline{#1}}
\newcommand{\Ind}{[\m{0},\m{p-1}]}
\newcommand{\cay}{{\sf Cay}}
\newcommand{\laa}{\langle\!\langle}
\newcommand{\raa}{\rangle\!\rangle}
\newcommand{\lla}{\langle\!\langle}
\newcommand{\rra}{\rangle\!\rangle}
\newcommand{\A}{{\cala A}}
\newcommand{\aut}[1]{{\sf Aut}(#1)}
\newcommand{\bsets}[1]{\Bsets{#1}}
\newcommand{\myplus}{\biguplus} 
\newcommand{\Stab}{{\sf Stab}}
\newcommand{\spn}{{\sf Span}}
\newcommand{\Ss}{{\cala S}}
\newcommand{\mymid}{\,\,\vline\,\,}
\newcommand{\sfM}{{\sf M}}
\newcommand{\cM}{{\mathcal M}}
\newcommand{\F}{{\mathbb F}}
\newcommand{\cR}{{\mathcal R}}
\newcommand{\cP}{{\mathcal P}}
\newcommand{\cF}{{\mathcal F}}
\newcommand{\cL}{{\mathcal L}}
\newcommand{\cT}{{\mathcal T}}
\newcommand{\cB}{{\mathcal B}}
\newcommand{\cC}{{\mathcal C}}
\newcommand{\cD}{{\mathcal D}}
\newcommand{\e}{{\sf e}}
\newcommand{\fix}{{\sf Fix}}
\newcommand{\trace}{{\rm trace}}

\newcommand{\sym}[1]{{\sf Sym }(#1)}
\newcommand{\cel}{{\rm Cel}}
\newcommand{\lm}{\lambda}
\newcommand{\SP}{{\rm Sp}}
\newcommand{\ot}[1]{{\bf O}^\theta(#1)}
\newcommand{\wt}[1]{\widetilde{#1}}
\newcommand{\wh}[1]{\widehat{#1}}
\newcommand{\btwn}[2]{{}_{#1}\!S_{#2}}
\newcommand{\haho}{half-homogeneous}
\newcommand{\coco}{coherent configuration}
\newcommand{\cocos}{coherent configurations}
\newcommand{\irr}{{\rm Irr}}
\newcommand{\rk}{{\rm rk}}
\newcommand{\cS}{{\mathcal S}}
\newcommand{\ea}{\Z_2^3}
\newcommand{\sg}[1]{\langle{#1}\rangle}
\newcommand{\ovr}[1]{\overline{#1}}
\newcommand{\tr}{T}
\newcommand{\vpphi }{\varphi}
\newcommand{\ord}{{\rm ord}}
\newcommand{\cyc}{\ cyclotomic}
\newcommand{\N}{{\sf N}}
\newcommand{\cA}{{\mathcal A}}
\newcommand{\Ker}{{\rm Ker}}
\newcommand{\End}{{\rm End}}
\newcommand{\Fr}{{\rm Fr}}
\newcommand{\fq}{{\mathbb F}_q}
\newcommand{\sig}{\sigma }
\newcommand{\cN}{{\mathcal N}}
\newcommand{\veps}{\varepsilon}
\newcommand{\cd}{{\rm c.d.}}
\newcommand{\mtrx}[4]{\left(\begin{array}{cc} #1 & #2 \\ #3 & #4\end{array}\right)}
\newcommand{\irreps}{ irreducible representations }
\newcommand{\core}{{\sf core}}
\def\f{\EuScript}
\newcommand{\comment}[1]{}
\newcommand{\CI}{CI${}^{(2)}$}
\newcommand{\fA}{\mathfrak{A}}
\newcommand{\fB}{\mathfrak{A}}
\newcommand{\fS}{\mathfrak{S}}
\newcommand{\fT}{\mathfrak{T}}
\newcommand{\iso}{\mathsf{Iso}}

\def\Zq{\mathbb{Z}_q}
\def\Zp{\mathbb{Z}_p^3}
\def\Zpq{\mathbb{Z}_q \times \mathbb{Z}_p^3 }
\def\hatZpq{\hat{\mathbb{Z}}_q \times \hat{\mathbb{Z}}_p^3}
\def\ringZpq{\mathring{\mathbb{Z}}_q \times \mathring{\mathbb{Z}_p^3}}
\def\stb{,\ldots ,}
\def\bfs{\bfseries}

\title{The Cayley isomorphism property for $\Z_p^3 \times \Z_q$}
\author[1]{Mikhail Muzychuk}
\author[2]{G\'abor Somlai}
\affil[1]{Ben Gurion University of the Negev}
\affil[2]{E\"otv\"os Lor\'and University, Departement of Algebra and Number Theory}
\date{}
\setcounter{Maxaffil}{0}
\maketitle

\begin{abstract}
For every pair of distinct primes $p$, $q$ we  prove that $\mathbb{Z}_p^3 \times \mathbb{Z}_q$  is a CI-group with respect to binary relational structures. 
\end{abstract}

\section{Introduction}\label{intro}
Let $H$ be a finite group and $S$ a subset of $G$. The \textit{Cayley digraph}
$Cay(H,S)$ is defined by having the vertex set $H$ and $g$ is adjacent
to $h$ if and only if $g h^{-1} \in S$. The set $S$ is called the
\textit{connection set} of the Cayley graph $Cay(H,S)$. 
An undirected Cayley  digraph will be referred to as a Cayley graph. Recall that a Cayley digraph
$Cay(H,S)$ is undirected if and only if $S=S^{-1}$, where $S^{-1} =
\left\{ \, s^{-1}  \mid s \in S \, \right\}$.
Every right multiplication via elements of $H$ is an automorphism of
$Cay(H,S)$, so the automorphism group of every Cayley graph over $H$
contains a regular subgroup denoted by $\hat{H}$ isomorphic to $H$. Moreover, this property characterises the Cayley graphs of $H$.

By a \textit{binary Cayley structure} (or a \textit{colored Cayley graph}) over $H$ we mean an ordered tuple $(\cay(H,S_1),...,\cay(H,S_r))$ of Cayley graphs, which we will always abbreviate as 
$\cay(H,(S_1,...,S_r))$. An isomorphism between two tuples 
$\cay(H,(S_1,...,S_r))$ and $\cay(H,(T_1,...,T_r))$ is a permutation 
$f\in\sym{H}$ satisfying $\cay(H,S_i)^f = \cay(H,T_i),i=1,...,r$. 
With this definition, the automorphism group of $\cay(H,(S_1,...,S_r))$ coincides with $\bigcap_{i=1}^r \cay(H,S_i)$.

It is clear that every automorphism $\mu$ of the group $H$ induces an isomorphism between $\cay(H,(S_1,...,S_r))$ and 
$\cay(H,(S_1^{\mu },...,S_r^{\mu }))$.
Such an isomorphism is called a \textit{Cayley isomorphism}. A colored Cayley digraph $Cay(G,\fS)$, where $\fS \in \mathcal{P}(H)^r$ has the \textit{CI-property} (or is a \textit{colored CI-digraph}) if, for each $\fT\in\mathcal{P}(H)^r$ the colored Cayley digraph $\cay(H,\fT)$ is isomorphic to
$Cay(G,\fS)$ if and only if they are \textit{Cayley isomorphic}, i.e.  there is an automorphism 
$\mu$ of $H$ such that $\fS^{\mu} =\fT$. In this case we say that $H$ has the CI-property for binary relational structures, or, it is a CI$^{(2)}$-group.
Furthermore, a group $H$ is called a DCI-group if every Cayley digraph of $H$ is a CI-digraph and it is called a CI-group if every undirected Cayley digraph of $H$ is a CI-graph. 

Investigation of the isomorphism problem of Cayley graphs started with \'Ad\'am's conjecture \cite{adam}. Using our terminology, it was conjectured that every cyclic group is a DCI-group. This conjecture was first disproved by Elspas and Turner \cite{ElspasTurner} for directed Cayley graphs of $\mathbb{Z}_8$ and for undirected Cayley graphs of $\mathbb{Z}_{16}$.

Analyzing the spectrum of circulant graphs Elspas and Turner \cite{ElspasTurner}, and independently Djokovi\'{c} \cite{Djokovic} proved that every cyclic group of order $p$ is a CI-group if $p$ is a prime. Also, a lot of research was devoted to the investigation of circulant graphs. One important result for our investigation is that $\mathbb{Z}_{pq}$ is a DCI-group for every pair of primes $p<q$. This result was first proved by Alspach and Parsons \cite{AlspachParsons} and independently by P\"oschel and Klin \cite{KlinPochel4} using the theory of Schur rings, and also by Godsil \cite{godsil}. Finally, Muzychuk \cite{Muzysquarefree,Muzycyclic1} proved that a cyclic group $\mathbb{Z}_n$ is a DCI-group if and only if $n=k$ or $n=2k$, where $k$ is square-free. Furthermore, $\mathbb{Z}_n$ is a CI-group if and only if $n$ is as above or $n=8,9,18$.

It is easy to see that every subgroup of a (D)CI-group is also a (D)CI-group so it is natural to investigate $p$-groups which are the Sylow $p$-subgroups of a finite group. Babai and Frankl \cite{babaifrankl} proved that if $H$ is a $p$-group, which is a CI-group, then $H$ can only be elementary abelian $p$-group, the quaternion group of order $8$ or one of a few cyclic groups $\mathbb{Z}_4$, $\mathbb{Z}_8$, $\mathbb{Z}_9$ or $\mathbb{Z}_{27}$.
The known results about cyclic groups show that $\mathbb{Z}_{27}$ is not a CI-group and $\mathbb{Z}_9,\mathbb{Z}_8$ are not DCI-group. Babai and Frankl also asked whether every elementary abelian $p$-group is a (D)CI-group.

The cyclic group of order $p$, which is a CI-group, can also be considered as an elementary abelian $p$-group of rank $1$. Currently, the best general result is due to Feng and Kov\'acs \cite{Zp5} who proved that $\mathbb{Z}_p^5$ is a CI-group for every prime $p$. The proof using elementary tools for $\Z_p^4$ is due to Morris \cite{morris}. It was shown by Somlai \cite{S11} that $\Z_p^r$ is not a DCI-group if $r\geq 2p+3$.

Severe restriction on the structure of DCI-groups was given by Li and Praeger and then a more precise list of candidates for DCI-groups was given by Li, Lu and P\'alfy \cite{liluppp}.
New family of CI-groups was found by Kov\'acs and Muzychuk \cite{KovacsMuzychuk}, that is, $\mathbb{Z}_{p}^2 \times \mathbb{Z}_q$ is a DCI-group for every prime $p$ and $q$. One example of DCI-groups connected to the question treated in this paper is $\mathbb{Z}_2^3 \times \mathbb{Z}_p$, see \cite{Dobson}.
It was also conjectured in \cite{KovacsMuzychuk}, that the direct product of DCI-groups of coprime order is a DCI-group\footnote{The cited paper deals in fact with DCI-groups while it talks about CI-groups.}. Note that the conjecture is not true for CI-groups as it was shown recently by T. Dobson \cite{D18}. Dobson also proved that the product of relatively prime order elementary abelian groups DCI-groups is a DCI-group by posing serious assumption on the prime divisors of the order of the group \cite{Dobson}. 

In this paper we prove the following result which supports this conjecture.
\begin{Thm}\label{fotetel}
For every pair of primes $p \ne q$, the group $\mathbb{Z}_{p}^3 \times \mathbb{Z}_q$ is a DCI-group.
\end{Thm}
In fact we prove here a more general fact: the above group is a CI${}^{(2)}$-group.
Our paper is organized as follows. In section \ref{sec1} we introduce the basic notation from Schur rings theory which are needed in this paper. In section \ref{sec2} we prove some general results about Schur rings over abelian groups of special order.  Finally, Section \ref{nagy} contains the proof of Theorem \ref{fotetel}.

\section{Schur rings}\label{sec1}

The result below is a direct consequence of Babai's lemma \cite{babai}.
\begin{lem}\label{babai}
A colored Cayley graph $\cay(H,\fS), \fS\in\mathcal{P}(H)^r$ has the CI-property if and only if any $H$-regular subgroup\footnote{An $H$-regular subgroup is any regular subgroup of the symmetric group isomorphic to $H$.} of the full automorphism group $\aut{\cay(H,\fS)}$ is conjugate to $\hat{H}$ inside $\aut{\cay(H,\fS)}$.
\end{lem}
According to this result, in order to prove the CI-property for binary Cayley structures, it is sufficient to go through the whole set of
automorphism groups of all colored Cayley graph over $H$. This could be done using the method of Schur rings. 
Let $G:=\aut{\cay(H,\fS)},\fS=(S_1,...,S_r)$ denote the full automorphism group of a colored digraph $\cay(H,\fS)$. Its intersection with $Aut(H)$ will be denoted as $Aut_H(\cay(H,\fS))$. Let us order the orbits of $G_1$ in arbitrary way, say $O_1,...,O_t$. Since 
$\aut{\cay(H,(S_1,...,S_r))} = \aut{\cay(H,(O_1,...,O_t))}$, we have to analyze only those colored Cayley graphs which correspond to overgroups $G\leq\sym{H}$ of $\hat{H}$. It turns out that these colored Cayley graph are closely related to Schur rings.

\subsection{Schur rings over finite groups}

We start with the basic definitions \cite{wielandt}.
Given a group $H$, we denote its group algebra over the rationals as
 $\mathbb{Q}[H]$. If $S\subseteq H$, then by $\underline{S}$ we denote the element  $\sum_{s\in S} s\in\mathbb{Q}[H]$. Following \cite{wielandt} we call elements of this type {\it simple quantities}.

A subalgebra $\mathfrak{A}$ of the group ring $\mathbb{Q}[H]$ is called a \textit{Schur ring}, an \textit{S-ring} for short, if it satisfies the following conditions.
\begin{enumerate}
\item
There exists a partition $\mathcal{T}=\{ T_0, T_1, \ldots ,T_l\}$ of $H$ such that
$\mathfrak{A}$ is generated as a vector space by the elements of the following form: $\underline{T}=\sum_{t \in T}t$.
\item
 $T_0=\{e\}$.
\item For each $0\le i\le l$ the subset $T_i^{(-1)}=\{t^{-1} \mid t \in T_i \}$ belongs to $\mathcal{T}$.
\end{enumerate}
The elements of the partition $\mathcal{T}$ are called \textit{basic sets} of $\mathfrak{A}$ and $\underline{T}_i$'s are called \textit{basic quantities}. 
In what follows the notation $\bsets{\mathfrak{A}}$ will stand for $\mathcal{T}$ and any partition satisfying the above conditions will be referred to as a \textit{Schur partition}.


One of the most natural examples of Schur rings are the \textit{transitivity modules}. Let $\hat{H}\leq\sym{H}$ be the right regular representation of a finite group $H$ and $G\leq\sym{H}$ its overgroup, i.e. $\hat{H}\leq G$. Then the orbits of the stabilizer $G_1$
are the basic sets of Schur ring over $H$ \cite{Schur}. Such a Schur ring will be denoted by $V(G,H)$. If $G = \hat{H} M$ for some $M\leq\aut{H}$, then the Schur ring $V(G,H)$ is called \textit{cyclotomic}. In this case, the basic sets of $V(G,H)$ coincide with the orbits of $M$. 


Every Schur partition (equivalently every S-ring) $\cT=\{ \und{T_0},...,\und{T_d}\}$ gives rise to an association scheme $\cay(H,\cT)$ whose basic graphs are the Cayley graphs $\cay(H,T), T\in\cT$. 
 Two Schur partitions (Schur rings)  $\mathfrak{A} \subseteq \mathbb{Q}[H],\mathfrak{B}\subseteq \mathbb{Q}[F]$ are called \textit{(combinatorially) isomorphic} if the corresponding association schemes are isomorphic, i.e.  there exists a bijection $f:H\rightarrow F$
which maps the basic Cayley graphs $\cay(H,S)$ for every  $S\in\bsets{\mathfrak{A}}$  bijectively onto the set 
$\{\cay(H,T)\}_{T\in\bsets{\mathfrak{B}}}$. The bijection $f$ is called a \textit{combinatorial isomorphism} between
$\mathfrak{A}$ and $\mathfrak{B}$. The isomorphism $f$ is called \textit{normalized} if $f(1_H)=1_F$. If $f$ is a 
normalized isomorphism between $\mathfrak{A}$ and $\mathfrak{B}$, then $\bsets{\mathfrak{A}}^f =\bsets{\mathfrak{B}}$.

We denote by $\iso(\mathfrak{A},\mathfrak{B})$ the set of all combinatorial isomorphisms between 
$\mathfrak{A},\mathfrak{B}$ and by $\iso_1(\mathfrak{A},\mathfrak{B})$ its subset consisting of the normalized ones. 
It is easy to see that $\iso(\mathfrak{A},\mathfrak{B}) = \hat{H}\iso_1(\mathfrak{A},\mathfrak{B}) = \iso_1(\mathfrak{A},\mathfrak{B})\hat{F}$.

Note that  $\iso(\mathfrak{A},\mathfrak{B})$ is empty if and only if  $\mathfrak{A},\mathfrak{B}$ are not combinatorially isomorphic. 

In what follows 
we write $\iso(\mathfrak{A},*)$ for the union of $\iso(\mathfrak{A},\mathfrak{B})$, where the second argument runs among all S-rings over the group $H$. As before,\\ $\iso(\mathfrak{A},*) = \hat{H}\iso_1(\mathfrak{A},*) = \iso_1(\mathfrak{A},*)\hat{H}$.

Two S-rings $\mathfrak{A}\subseteq \mathbb{Q}[H]$ and $\mathfrak{B}\subseteq \mathbb{Q}[F]$ are \textit{Cayley isomorphic} if there exists
a group isomorphism $\varphi:H\rightarrow F$ such that $\varphi(\mathfrak{A})=\mathfrak{B}$. Note that Cayley isomorphic S-rings are always combinatorially isomorphic but not vice versa.

An S-ring $\fA$ is a \textit{CI}-S-ring if for any S-ring $\fB'\subseteq\mathbb{Q}[H]$ and arbitrary $f\in\iso_1(\fA,\fB')$ there exists 
$\varphi\in\aut{H}$ such that $f(S)=\varphi(S)$ for all $S\in\bsets{\fA}$. It follows directly from the definition that an S-ring $\mathfrak{A}$ is CI-S-ring if and only if $\iso_1(\mathfrak{A},*)=\aut{\mathfrak{A}}\aut{H}$.

As an application of Babai's lemma \cite{babai} we have the following statement \cite{HirasakaMuzychukchuk}.
\begin{prop}
Let $\Gamma:=\cay(H,\Sigma)$ be a colored Cayley graph over $H$ and $G:=\aut{\Gamma}$. The following are equivalent
\begin{enumerate}
\item[{\rm (a)}] $\Gamma$ has the CI-property;
\item[{\rm (b)}] any $H$-regular subgroup of $G$ is conjugate to $\hat{H}$ in $G$;
\item[{\rm (c)}] the transitivity module $V(H,Aut(\Gamma)_1)$ is a CI-S-ring.
\end{enumerate}
\end{prop}

This implies the following result.
\begin{Thm}\label{290519a} A group $H$ has a CI-property  for binary relational structures (\CI-group, for short) if and only if every transitivity module over $H$ is a CI-S-ring.
\end{Thm}
Thus one has to check all transitivity modules over the group $H$. To reduce the number of checks we use the following partial order on the set $\Sup{\hat{H}}$ consisting of all overgroups of $\hat{H}$.

Given two overgroups $X,Y\in\Sup(\hat{H})$, we write $X\preceq_{\hat{H}} Y$ if any $H$-regular subgroup of $Y$ may be conjugated into $X$ by an element of $Y$, i.e.
$$
\forall_{g\in\sym{H}}:\ \hat{H}^g\leq Y\implies \exists y\in Y:
(\hat{H}^g)^y\leq X.
$$
One can easily check that $\preceq_{\hat{H}}$ is a partial order on the set of all overgroups of $\hat{H}$.

The statement below allows us to consider transitivity modules of $\prec_{\hat{H}}$-minimal groups only.
\begin{prop}\label{110619a} Let $G_1\leq G_2$ be two overgroups of 
$\hat{H}$ and $\mathfrak{A}_i:=V(G_i,H)$ their transitivity modules.
Then $\mathfrak{A}_1\supseteq\mathfrak{A}_2$. If $G_1\preceq_{\hat{H}}\aut{\mathfrak{A}_2}$ and $\mathfrak{A}_1$ is CI, then $\mathfrak{A}_2$ is also a CI-S-ring.
\end{prop}
Sylow's theorem shows that if $H$ is a $p$-group, then any $\preceq_{\hat{H}}$-minimal overgoup of $\hat{H}$ is a $p$-group. In this case we are left to investigate transitivity modules
whose basic sets have a $p$-power cardinality. These Schur rings are called \textit{$p$-Schur rings}.

\subsection{Structural properties of Schur rings}
As before, $H$ is a finite group and $\mathbb{Q}[H]$ is its group algebra. 
For an element of the group algebra $T=\sum_{g\in H} a_g g$ let $T^{(m)}=\sum_{g\in H} a_g g^m$. Two Lemmas below are taken from \cite{wielandt}.
\begin{lem}\label{lemgcd}
Let $\mathfrak{A}$ be an S-ring over an abelian group $H$. If $gcd(m,|H|)=1$, then $T^{(m)} \in \mathfrak{A}$ for every $T \in \mathfrak{A}$.
\end{lem}
A similar statement holds of $m$ divides $|H|$.
\begin{lem}\label{lempower,congruence}
Let $\underline{T}$ be a simple quantity and $m$ a prime divisor of $|G|$ and let $\underline{T}^{(m)} = \sum_{g \in G}a_g g$. Then $\sum_{a_g \not\equiv 0 \pmod{m}} g \in \mathfrak{A}$.
\end{lem}
A Schur ring $\mathfrak{A}$ is called \textit{imprimitive} if for some non-trivial subgroup $L \le H$, the basic set $\underline{L}$ is an element of $\mathfrak{A}$. Such a subgroup is called an \textit{$\mathfrak{A}$-subgroup}. If $T$ is an $\mathfrak{A}$-set, then we may define its \textit{radical} $\rad(T)=\{g \in T \mid ~ Tg=T\}$.
It is well known that the radical of an $\mathfrak{A}$-set $T$ is an $\mathfrak{A}$-subgroup \cite{wielandt}.

We say that $\mathfrak{A}$ is \textit{primitive} if the only $\mathfrak{A}$-subgroups are $1$ and $H$.
For an $\mathfrak{A}$-subgroup $U$ one can define $\mathfrak{A}_U$ as the restriction of $\mathfrak{A}$ to $U$ spanned by $\mathfrak{A}$-sets contained in $U$. For a pair of $\mathfrak{A}$-subgroups $L \mathrel{\unlhd} U$
we define $\mathfrak{A}_{U/L}$ as a subring of $\Z[U/L]$ spanned by $\{ \underline X^{\pi} ~\mid~ X \subset U, ~x \in \bsets{\mathfrak{A}}\}$, where $\pi$ denotes the canonical epimorphism from $U$ to $U/L$.

We say that the Schur ring $\mathfrak{A}$ is a \textit{generalized wreath product} if there exists $\mathfrak{A}$-subgroups $L\leq U$  such that $L$ is a normal subgroup in $H$ and every basic set outside of $U$ is the union of $L$-cosets.
Such a wreath product is called \textit{trivial} if $L=\{e\}$ or $U=H$.

Let $K$ and $L$ be two $\mathfrak{A}$-subgroups. We say that $\mathfrak{A}$ is the \textit{star product} of $\mathfrak{A}_K$ and $\mathfrak{A}_L$ (or $\mathcal{A}$ \textit{admits a star decomposition}) if the following conditions hold:
\begin{enumerate}
\item $K \cap L \unlhd L$
\item each basic set $T$ of $\mathfrak{A}$ with $T \subseteq (L\setminus K)$ is the union of $K \cap L$-cosets
\item\label{itemstarc}
for each basic set $T \subseteq H\setminus (K \cup L)$ there exists
$R, S \in \bsets{\mathfrak{A}}$, where $R \subseteq K$, $S \subseteq L$ such that $T = RS$.
\end{enumerate}
Note that it is enough to verify for \ref{itemstarc} that $R$ and $S$ are $\mathfrak{A}$-sets. 

In this case we write $\mathfrak{A} =\mathfrak{A}_K \star \mathfrak{A}_L$. A star-decomposition is called \textit{trivial} if $K = 1 \mbox{ or } H$.

The theorems below provide us sufficient conditions for these product to have the CI-property. Although both of the statements were originally proved for elementary abelian groups, their proofs work for a more general class of groups, namely: the abelian groups with elementary abelian Sylow subgroups. In what follows we refer to these groups as $\mathcal{E}$-groups. 

\begin{Thm}[\cite{HM}]\label{thmstar}
Let $H$ be an $\mathcal{E}$-group and let $G \le Sym(H)$ be an overgroup of $\hat{H}$. If $V(H,G_1)$ admits 
a nontrivial star-decomposition with CI-factors, then $V(H,G_1)$ is a CI-S-ring.
\end{Thm}
In the case of generalized wreath product we have the following result.

\begin{Thm}[\cite{KR}]\label{kovacsryabov}
Let $H$ be an $\mathcal{E}$-group and let $G \le Sym(H)$ be an overgroup of $\hat{H}$. Assume that $\mathfrak{A}:=V(G,H_1)$ 
is a non-trivial generalized wreath product with respect
to $\mathfrak{A}$-subgroups $\{e\} \ne L\le U \ne H$. Assume that $\mathfrak{A}_U$ and $\mathfrak{A}_{H/L}$ are CI-$S$-rings and $Aut_{U/L}(\mathfrak{A}_{U/L})= Aut_{U}(\mathfrak{A}_{U})^{U/L}Aut_{H/L}(\mathfrak{A}_{H/L})^{U/L} $. Then $\mathfrak{A}$ is a CI-S-ring.
\end{Thm}

\section{Schur rings over abelian group of non-powerful order}\label{sec2}

Recall that a number $n$ is call {\it powerful} if $p^2$ divides $n$ for every prime divisor $p$ of $n$. In this section and in what follows we assume that $H$ is an abelian group of non-powerful order, i.e. there exists a prime divisor $q$ of $|H|$ such that $|H| = n q$ where $n$ is coprime to $q$. In what follows we call such $q$ a {\it simple} prime divisor of $|H|$. We assue that $q > 2$.
 
Let $P$ and $Q$ denote the unique subgroups of $H$ of orders $n$ and $q$, respectively and let $Q^{\#}=Q \setminus\{1\}$.
Let $e$ be the exponent of $P$. The group $\Z_{e q}^*\cong\Z_{e}^*\times \Z_q^*$ acts on $H$ via raising to the power as $h\mapsto h^t$, where $t\in \Z_{e q}^*$. Denote $M_q:=\{t\in\Z_{eq}^*\,|\, t\equiv\, 1 \pmod{e} \}$. Clearly $M_q\cong\Z_q^*$.

Every element $h\in H$ has a unique decomposition into the product $h = h_{q'} h_q$ where $h_{q'}\in P$ and $h_q\in Q$. Notice that two elements $h,f\in H$ belong to the same $Q$-coset if and only if $h_{q'}=f_{q'}$. Let $q^*\in \Z_{e q}^*$ be an element satisfying $q^*q\equiv\, 1 \pmod{ e}$ and  $q^*\equiv\, 1\pmod{p}$. Then $h_p = h^{qq^*}$.

Given a subset $T\subseteq H$. We write $T_{q'}$ for the set $\{h_{q'}\,|\, h\in T\}$. Notice that $T_{q'}$ is always contained in $P$.
We always have the decomposition $T = \bigcup_{s\in T_p} s R_s$ where $R_s:=s^{-1}T \cap Q$.

In what follows $\mathfrak{A}$ stands for a non-trivial S-ring over $H$. Let $P_1$ is the maximal $\mathfrak{A}$-subgroup contained in $P$ while $Q_1$ is the minimal $\mathfrak{A}$-subgroup which contains $Q$.

We start with the following statement which is a direct consequence of 
Theorem 25.4 \cite{wielandt}.
\begin{prop}\label{W} Let $H$ be an abelian group. If $|H|$ has a simple prime divisor, then any primitive S-ring over $H$ is trivial. 
\end{prop}



The statement below describes the structure of $M_q$-invariant basic sets.
\begin{prop}\label{250115a} Let $T$ be a basic set of
$\mathfrak{A}$ which is $M_q$-invariant. Denote $S:=T_{q'}$. There
exists a partition\footnote{Notice that some of its parts may be
empty} $S = S_1\cup S_{-1} \cup S_0$ such that $T = S_1 \cup
S_{-1}Q^\# \cup S_0 Q$ and $S_1,S_{-1}$ are $\mathfrak{A}$-subsets (not necessarily basic). In addition the sets $S_1,S_{-1}$ and $S_0$ satisfy the following conditions
\begin{enumerate}
\item If $S_1\neq \emptyset$, then $S_{-1}=S_0=\emptyset$ and $T\subseteq P_1$;
\item If $S_1 = \emptyset$ and $S_{-1} \neq \emptyset$, then $T = S_{-1} (Q_1\setminus P_1)$;
\item\label{itemc} If $S_1 = S_{-1}=\emptyset$, then $Q_1 T = T$.
\end{enumerate}
\end{prop}
\begin{proof}
Write $T = \bigcup_{s\in S} s R_s$ where $R_s:=s^{-1}T\cap Q$. Since $T$ is $M_q$-invariant, the sets $R_s$ are $\Z_q^*$-invariant. Therefore $R_s \in\{\{1\}, Q^\#,Q\}$. Now the sets
$$
S_1:=\{s\,|\,R_s=\{1\}\}, S_{-1}:=\{s\,|\,R_s=Q^\#\}, S_{0}:=\{s\,|\,R_s=Q\}
$$
produce the required partition. Raising the simple quantity $\und{T} = \und{S_1}+\und{S_{-1}}\cdot\und{Q^\#} + \und{S_0}\cdot\und{Q}$ to the $q$-th power modulo $q$ we obtain
$$
\und{T}^q \equiv (\und{S_1})^q - (\und{S_{-1}})^q
\equiv (\und{S_1^{(q)}}) - (\und{S_{-1}^{(q)}}) \pmod{q}.
$$
Now by Schur-Wielandt principle $S_1^{(q)},S_{-1}^{(q)}$ are $\mathfrak{A}$-subsets. Applying $q^*$ we conclude that $S_1$ and $S_{-1}$ are $\mathfrak{A}$-subsets too.

If $S_1\neq\emptyset$, then $S_1 = T$ because $T$ is basic and $S_1$ is nonempty $\mathfrak{A}$-subset contained in $T$. Hence $S_{-1}=S_0=\emptyset$.

Assume now that $S_1=\emptyset$ and $S_{-1}\neq\emptyset$. Since $Q_1\setminus P_1 = Q_1\setminus (Q_1\cap P_1)$ is an $\mathfrak{A}$-subset
which contains $Q^\#$, we conclude that $S_{-1} (Q_1\setminus P_1)$ is an $\mathfrak{A}$-subset which intersects $T$ non-trivially (the part $S_{-1} Q^\#$ is in common). Therefore $S_{-1} (Q_1\setminus P_1)\supseteq T$.

The union $S_{-1} \cup T = (S_{-1}\cup S_0)Q$ is an $\mathfrak{A}$-subset
the radical of which contains $Q$. Therefore, by the minimality of $Q_1$, we have $Q_1\leq rad(S_{-1}\cup T)$.
This implies $Q_1 S_{-1} \cup Q_1 T = S_{-1}\cup T $ so $ S_{-1} Q_1  \subseteq S_{-1}\cup T$.
Thus $T\subseteq S_{-1} (Q_1\setminus P_1)  \subseteq S_{-1}\cup T$. If $S_{-1} (Q_1\setminus P_1)\cap S_{-1} \neq \emptyset$, then $s t =s'$ for some $s,s'\in S_{-1}$ and $t\in Q_1\setminus P_1$.
But in this case we would  obtain $t =s's^{-1}\subseteq S_{-1} S_{1}^{(-1)}\subseteq P_1$, a contradiction.
Hence $S_{-1} (Q_1\setminus P_1)\cap S_{-1} = \emptyset$ implying that $T = S_{-1} (Q_1\setminus P_1)$.

If $S_1=S_{-1}=\emptyset$. then $T=S_0Q$ so $ rad(T)$ contains $Q$ By the minimality of $Q_1$
we have $Q_1 \le  rad(T)$ so $Q_1T=T$.
\end{proof}

\begin{cor}\label{wedge} $\mathfrak{A}$ is a generalized wreath product with respect to $Q_1$ and $P_1Q_1$.
\end{cor}
\begin{proof} There is nothing to prove if $Q_1 P_1 =H$. So, in what follows we assume that $Q_1P_1\neq H$.

We have to show that $Q_1T = T$ holds for each $\mathfrak{A}$-basic set $T$ outside of $P_1Q_1$. Let $T$ be such a basic set, that is, $T\cap P_1Q_1=\emptyset$. 

If $T$ contains a $q'$-element, then $T$ is $M_q$ -invariant, and therefore, $T$ fits one of the cases described in Proposition~\ref{250115a}. The cases (a) and (b) contradict $T\cap P_1Q_1=\emptyset$, since in both of them $T\subseteq P_1Q_1$.
Therefore the case \ref{itemc} of Proposition~\ref{250115a} occurs and $TQ_1=T$, as required.

It remains to show that every basic $\mathfrak{A}$-set disjoint with $P_1Q_1$ contains a $q'$-elements. Assume that there exists one, say $T$,
which does not contain a $q'$-element. Denote $R:=T_{q'}$. Then $T =\cup_{h\in R} h Q_h$ where $Q^\#\supseteq Q_h\neq\emptyset$. Then by Lemma \ref{lempower,congruence} $T^{(q)}=R^{(q)}$ is an $\mathfrak{A}$-set, implying that $R^{(q)}\subseteq P_1$ and $R\subseteq P_1$. Again we have $T\subseteq RQ\subseteq P_1Q_1$, contrary to a choice of $T$.
\end{proof}

\subsection{ The structure of the section $\mathfrak{A}_{P_1Q_1}$}

In what follows we abbreviate $H_1:=P_1 Q_1$ and $\mathfrak{A}_1:=\mathfrak{A}_{H_1}$. We start with the following simple statement.

\begin{prop}\label{050618a} $P_1$ is an $\mathfrak{A_1}$-maximal subgroup.
\end{prop}
\begin{proof} Let $\tilde{P_1}$ denote a proper $\mathfrak{A}_1$-maximal subgroup which contains $P_1$. If $q$ divides $\tilde{P_1}$, then $Q_1$ is contained in $\tilde{P_1}$ implying $ P_1 Q_1 \leq \tilde{P_1}= H_1$, a contradiction. Hence $\tilde{P_1}$ is a $p$-group, which is an $\mathfrak{A}_1$-subgroup. Therefore, $\tilde{P_1}=P_1$.
\end{proof}
Since $P_1$ is an $\mathfrak{A}_1$-maximal subgroup, the quotient S-ring is primitive. By Wielandt's Theorem either the quotient S-ring has rank two or $H_1/P_1$ is of prime order. In the latter case, 
$|H_1/P_1|=q$.

\begin{prop}\label{110519a} If the quotient S-ring $\mathfrak{A}_1/P_1$ has rank two, then $\mathfrak{A}_1=(\mathfrak{A}_1)_{Q_1}\star (\mathfrak{A}_1)_{P_1}$.
\end{prop}
\begin{proof} The quotient S-ring $\mathfrak{A}_1/P_1$ has rank two iff $TP_1 = H_1\setminus P_1$ holds for each basic set $T\in\Bsets{\mathfrak{A}_1}$ outside of $P_1 $. 

Assume first that $P_1\neq (H_1)_{q'}$. Pick an arbitrary $T\in\Bsets{\mathfrak{A}_1}$ with $T\cap P_1 = \emptyset$. Then  $TP_1=H_1\setminus P_1 \supseteq (H_1)_{q'}\setminus P_1$ implying $T\cap (H_1)_{q'}\neq\emptyset$. Thus $T$ contains $q'$-elements, and, therefore, is $M_q$-invariant and Proposition \ref{250115a} is applicable. 

The first case of the Proposition is not possible because $T \cap P_1= \emptyset$.

In the second case we obtain that $T$ is the product of two $\mathfrak{A}_1$-sets $S_{-1} \subset P_1$ and $Q_1 \setminus P_1 \subset Q_1$ so $T$ fits the definition of star decomposition.

Finally, if $Q_1T=T$, then $T$ is the union of $Q_1$-cosets. Since $P_1 Q_1=H_1$ we have that $P_1$ intersects every $Q_1$-coset. Hence $T \cap P_1 \ne \emptyset$, contradicting the choice of $T$.

Thus, we have proven that any basic set $T$ of $\mathfrak{A}_1$ disjoint to $P_1$ has a form $S (Q_1\setminus P_1)$ where $S\subseteq P_1$ is an $\mathfrak{A}_1$-subset so is a union of $P_1 \cap Q_1$-cosets. This immediately implies that $Q_1\setminus P_1$ is a basic set of $\mathfrak{A}_1$
and $\mathfrak{A}_1=(\mathfrak{A}_1)_{P_1}\star (\mathfrak{A}_1)_{Q_1}.$
\end{proof}
Note that the above argument implies that if $Q_1=H_1$, then $\mathfrak{A}$ is a wreath product with respect to $P_1$.

If $\mathfrak{A_1}/P_1$ is non-trivial, then   
$\mathfrak{A_1}/P_1$ is a non-trivial S-ring over a cyclic group of order $q$. In particular, $[H_1:P_1]=q$. 
Although the structure of S=rings over $C_q$ is known \cite{} we not need it, because for our purposes we need to settle the case when $\mathfrak{A_1}/P_1$ coincides with full group algebra.

From now on we will denote the cyclic group by $C_p$ in order to make the notation more readable. 
\begin{prop}\label{220519b}
If $\mathfrak{A_1}/P_1\cong\Z[C_q]$, then $\mathfrak{A}_1 = (\mathfrak{A_1})_{P_1}\star (\mathfrak{A_1})_{Q_1}$.
\end{prop}
\begin{proof} It follows from the assumption that cosets $hP_1,h\in Q^\#$ are $\mathfrak{A}_1$-subsets. Therefore $hP_1$ is partitioned into a disjoint union of basic sets yielding a partition $\Sigma_h$ of $P_1$: 
$$S\in\Sigma_h\iff hS\in\bsets{\mathfrak{A_1}}.$$
Since $M_q$ permutes basic sets and acts transitivey on $Q^\#$, 
the partitions $\Sigma_h$ does not depend on a choice of $h\in Q^\#$ by Lemma \ref{lemgcd}.
So, in what follows we write just $\Sigma$ without an index. 

Pick a basic set  $T$ outside of $P_1$. Then $T=hS$ for some $h \in Q^\#$ and $S\in\Sigma$. Now it follows from $\und{T}^q \equiv \und{S}^{(q)} ({\rm mod}\ q)$ that $S^{(q)}$ is an $\mathfrak{A}_1$-subset contained in $P_1$. Applying $q^*$ to $S^{(q)}$ we conclude that $S$ is an $\mathfrak{A}_1$-subset.

Since $\langle\und{T}\,|\,T\in \bsets{\mathfrak{A}_1}\land T\subseteq hP_1\rangle$ is an $(\mathfrak{A}_1)_{P_1}$-invariant subspace, the linear span $\und{\Sigma}:=\langle\und{S}\rangle_{S\in\Sigma}$ is an ideal of $\mathfrak{A_1}$. Let $S_1\in\Sigma$ be a class containing $1$. 

We claim that $\S_1$ is an $\mathfrak{A}_1$-subgroup and  every class of $\Sigma$ is a union of $S_1$-cosets. This will imply our claim.

Pick a basic set $T$ of $(\mathfrak{A_1})_{P_1}$ contained in $S_1$. Then $1$ appears in the product $\und{T}^{(-1)}\und{S_1}$ with coefficient $|T|$. Therefore $\und{S_1}$ appears $|T|$ times in this product. This implies $\und{T}^{(-1)}\und{S_1} = |T|\und{S_1}$ and, consequently, $T^{(-1)}S_1 = S_1$. Since this equality holds for any basic set $T$ contained in $S_1$, we conclude that $S_1^{(-1)} S_1=S_1$, hereby proving that $S_1$ is a subgroup of $P_1$.

Pick now an arbitrary $S\in\Sigma$. Then $\und{S}^{(-1)}\und{S}\in\und{\Sigma}$. The identity $1$ appear in the product $|S|$ times. Therefore $\und{S_1}$ appears in the product $\und{S}^{(-1)}\und{S}$ with coefficient $|S|$. Therefore $S$ is a union of $S_1$-cosets.

It is easy to see that $S_1h$ generates an $\mathfrak{A}_1$-subgroup, whose order is divisible by $q$ so it contains $Q_1$. On the other hand $S_1h$ is a basic set intersecting $Q$ nontrivially so it is contained in $Q_1$. Thus $S_1= Q_1 \cap P_1$, which gives that $\mathfrak{A}_1$ admits a star decomposition. 
\end{proof}

\section{Proof of the main result}\label{nagy}

In this section we show that every transitivity module over the group $H\cong C_p^3\times C_q, p\neq q$ are primes is a CI-S-ring. Since $q$ is a simple prime divisor of $|H|$, the structural results from the previous Section are applicable. We also keep the notation $P_1$ and $Q_1$ defined in Section~\ref{sec2}. 

For the rest of the section $\mathfrak{A}=V(G_1,H)$ is a transitivity module of an $\preceq_{\hat{H}}$-minimal subgroup $G$. 

In this section we prove the following 
\begin{Thm}\label{CIS} $\mathfrak{A}$ is a CI-S-ring.
\end{Thm}
Combining this result with Theorem~\ref{290519a} we obtain the main result of the paper.

\subsection{Proof of Theorem~\ref{CIS} in the case of $P_1Q_1\neq H$}

If $P_1Q_1\neq H$, then by Corollary~\ref{wedge} the S-ring $\mathfrak{A}$ is a non-trivial generalized wreath product of $\mathfrak{A}_{P_1Q_1}$ and $\mathfrak{A}_{H/Q_1}$. Therefore, the results of \cite{KR} are applicable.

Since $\overline{H}:=H/Q_1$ is an elementary abelian $p$-group, we may assume that the basic sets of $\overline{\mathfrak{A}}:=\mathfrak{A}/Q_1$ are of $p$-power length. 
Such a Schur ring is called a \textit{$p$-S-ring} and so $\overline{\mathfrak{A}}$ is a transitivity module of the quotient group $\overline{G}:=G^{H/Q_1}$. Since $G$ is $\preceq_{H}$-minimal, the group is $\overline{G}$ is a $\prec_{\overline{H}}$-minimal.

If $|P_1Q_1/Q_1|\leq p$, then $\mathfrak{A}_{P_1Q_1/Q_1}$ is the full group ring and we are done by Proposition 4.1 of \cite{KR}. Thus we may assume that
$|P_1Q_1/Q_1| = p^a$ with $a\geq 2$. Since $q$ divides $|P_1Q_1|$ and $P_1Q_1\neq H$,
we conclude that $|P_1|=p^2, |Q_1|=q$. Thus $\mathfrak{A}_{P_1Q_1/Q_1}\cong \Z[C_p] \wr \Z[C_p]$ since if
$\mathfrak{A}_{P_1Q_1/Q_1}\cong \mathbb{Z}[C_p^2]$ we may apply Proposition 4.1 of \cite{KR} and these are the only $p$-Schur rings over $\Z_p^2$. Further it follows from $|Q_1|=p$ that $\bar{H}\cong C_p^3$.

The S-ring $\mathfrak{A}_{\bar{H}}$ is a Schurian $p$-S-ring over the group $\overline{H}\cong C_p^3$. The classification of such S-rings is well-known \cite{HirasakaMuzychukchuk}. They are
$$
\begin{array}{lcl}
\mathfrak{B}_1 & = & \Z[C_p^3], \\
\mathfrak{B}_2 & = &\Z[C_p^2]\wr\Z[C_p],\\
\mathfrak{B}_3 & = &(\Z[C_p]\wr\Z[C_p])\otimes \Z[C_p],\\
\mathfrak{B}_4 & = &\Z[C_p]\wr\Z[C_p^2],\\
\mathfrak{B}_5 & = &\Z[C_p]\wr\Z[C_p]\wr\Z[C_p],\\
\mathfrak{B}_6 & = & V(C_p^3,(C_p^3 \rtimes \langle\alpha\rangle)_1)
\end{array}
$$
Here $\alpha\in\aut{C_p^3}$ is an automorphism of order $p$ which has $p$ fixed points. We can exclude the S-ring $\mathcal{B}_6$, because in this case the group $\overline{G}$ is not $\prec_{\overline{H}}$-minimal.

It follows from $\mathfrak{A}_{Q_1P_1/Q_1}\cong\Z[C_p]\wr\Z[C_p]$ that there exists
an $\overline{\mathfrak{A}}$-subgroup of order $p^2$ on which the induced Schur ring is isomorphic to $\Z[C_p]\wr\Z[C_p]$.
This excludes $\overline{\mathfrak{A}} \cong \mathfrak{B_1} \mbox{ or } \mathfrak{B_2}$.

It remains to settle the cases $\overline{\mathfrak{A}}\cong \mathfrak{B}_i, i=3,4,5$.

The inclusion $\mathrm{Aut}_{\overline{H}}(\overline{\mathfrak{A}})^{\overline{P_1}} \leq \mathrm{Aut}_{\overline{P_1}}(\mathfrak{A}_{\overline{P_1}})$
is trivial. To prove the inverse inclusion we  note that each of the $S$-rings $\mathfrak{B}_i,~i=3,4,5$ is cyclotomic.  In particular this implies that
$\mathrm{Aut}_{\overline{H}}(\overline{\mathfrak{A}})$ acts transitively on each basic set of $\overline{\mathfrak{A}}$. Therefore $\mathrm{Aut}_{\overline{H}}(\overline{\mathfrak{A}})^F$ is non-trivial whenever the induced S-ring $\overline{\mathfrak{A}}_F$ is non-trivial for any $\overline{\mathfrak{A}}$-subgroup $F$. This implies that $\mathrm{Aut}_{\overline{H}}(\overline{\mathfrak{A}})^{\overline{P_1}}$ is non-trivial.
Therefore, $p\leq |\mathrm{Aut}_{\overline{H}}(\overline{\mathfrak{A}})^{\overline{P_1}}|\leq 
|\mathrm{Aut}_{\overline{P_1}}(\mathfrak{A}_{\overline{P_1}})|.$

On the other hand, 
$\mathrm{Aut}_{\overline{P_1}}(\mathfrak{A}_{\overline{P_1}}) = \mathrm{Aut}_{C_p^2}(\Z[C_p]\wr\Z[C_p])$ is contained in a Sylow $p$-subgroup of $\mathrm{Aut}(C_p^2)\cong GL_2(p)$. Since the latter one has order $p$,
we conclude 
that $|\mathrm{Aut}_{\overline{P_1}}(\mathfrak{A}_{\overline{P_1}})|\leq p$ implying 
$\mathrm{Aut}_{\overline{H}}(\overline{\mathfrak{A}})^{\overline{P_1}} = 
\mathrm{Aut}_{\overline{P_1}}(\mathfrak{A}_{\overline{P_1}}).$

Therefore $\mathrm{Aut}_{\overline{H}}(\overline{\mathfrak{A}})^{\overline{P_1}} = \mathrm{Aut}_{\overline{P_1}}(\mathfrak{A}_{\overline{P_1}})$ and by Theorem \ref{kovacsryabov} of \cite{KR} the corresponding S-ring is CI.

\subsection{Proof of Theorem~\ref{CIS} in the case of $P_1Q_1 = H$.}

If the quotient $\mathfrak{A}_{H/P_1}$ is trivial, then by Proposition~\ref{110519a} we have $\mathfrak{A}=\mathfrak{A}_{P_1}\star \mathfrak{A}_{Q_1}$. Since both $P_1$ and $Q_1$ are $\mathcal{E}$-groups 
with at most three prime factors, they are CI${}^{(2)}$-groups. Therefore, $\mathfrak{A}_{P_1}$ and $\mathfrak{A}_{Q_1}$ are
CI-S-rings. By Theorem~3.2 in \cite{HM} $\mathfrak{A}$ is a CI-S-ring. Although Theorem Theorem~3.2 in \cite{HM} is about elementary abelian groups, its proof works for $\mathcal{E}$-groups without any change, see also \cite{KovacsMuzychuk}.

Assume now that $\mathfrak{A}_{H/P_1}$ is non-trivial. 
Since $P_1$ is maximal $\mathfrak{A}$-subgroup, the quotient $\mathfrak{A}_{H/P_1}$
is a non-trivial primitive S-ring over the group $H/P_1$.  
Since $H/P_1$ has a simple prime divisor $q$,  by Wielandt's theorem we conclude that $H/P_1\cong C_q$. Since $G$ is $\prec_H$-minimal,
its quotient $G^{H/P_1}$ is $\prec_{H/P_1}$-minimal too. Therefore 
$G^{H/P_1}\cong C_q$ and $\mathfrak{A}_{H/P_1}\cong\Z[C_q]$.
By Proposition~\ref{220519b} $\mathfrak{A}=\mathfrak{A}_{P_1}\star\mathfrak{A}_{Q_1}$. As before, we conclude that $\mathfrak{A}$ is a CI-S-ring.

\end{document}